\documentclass[10pt]{article}

\usepackage[margin=1in]{geometry}  
\usepackage{graphicx}              
\usepackage{amsmath}               
\usepackage{amsfonts}              
\usepackage{amsthm}                
\usepackage{amssymb}

\usepackage{comment}

\newtheorem{theorem}{Theorem}[section]
\newtheorem{lemma}[theorem]{Lemma}

\newtheorem{corollary}[theorem]{Corollary}
\newtheorem{definition}[theorem]{Definition}

\newtheorem{remark}[theorem]{Remark}

\begin{document}

\nocite{*}

\title{Subregular subalgebras and invariant generalized complex structures on Lie groups}

\author{Evgeny Mayanskiy}

\maketitle

\begin{abstract}
  We introduce the notion of a subregular subalgebra, which we believe is useful for classification of subalgebras of Lie algebras. We use it to construct a non-regular invariant generalized complex structure on a Lie group. As an illustration of the study of invariant generalized complex structures, we compute them all for the real forms of $G_2$. 
\end{abstract}

\section{Subregular subalgebras}

Let $\mathfrak{g}$ be a finite-dimensional complex Lie algebra, $k\geq 0$ an integer.

\begin{definition}
A subalgebra $\mathfrak{s}\subset \mathfrak{g}$ is called \textit{subregular} \textit{in codimension $k$} if $\mathfrak{s}$ is normalized by a codimension $k$ subalgebra of a Cartan subalgebra of $\mathfrak{g}$.

If $k\geq 1$, then $\mathfrak{s}\subset \mathfrak{g}$ is called \textit{subregular} \textit{strictly} \textit{in codimension $k$} if $\mathfrak{s}\subset \mathfrak{g}$ is subregular in codimension $k$, but is not subregular in codimension $k-1$.
\end{definition}

Note that any subalgebra $\mathfrak{s}\subset \mathfrak{g}$ is subregular (strictly) in codimension $k$ for some $k\in \{ 0,1,\ldots , \operatorname{rank}(\mathfrak{g})  \}$. Regular subalgebras, as defined in \cite{Chebotarev}, \cite{Dynkin}, are precisely those which are subregular in codimension $0$.\\

This notion may be useful for an explicit classification  of subalgebras of Lie algebras as in \cite{Mayanskiy}. In this note, we demonstrate how it can be applied to construction of invariant generalized complex structures on Lie groups.

\section{Invariant generalized complex structures}

Invariant generalized complex structures on homogeneous spaces were studied in \cite{Milburn} and \cite{Liana}. In particular, Alekseevsky, David and Milburn classified invariant generalized complex structures on Lie groups in terms of the so-called admissible pairs. We will review a part of their classification.\\

Throughout this section, $G_0$ denotes a finite-dimensional connected real Lie group, ${\mathfrak{g}}_0$ the (real) Lie algebra of $G_0$, $\mathfrak{g}=\mathfrak{g}_0 \otimes_{\mathbb R} \mathbb C$ its complexification and $\tau \colon \mathfrak{g} \to \mathfrak{g}$ the corresponding antiinvolution. If $\mathfrak{g}$ is semisimple and ${\mathfrak{h}}_0\subset {\mathfrak{g}}_0$ is a Cartan subalgebra, then $\mathfrak{h}=\mathfrak{h}_0 \otimes_{\mathbb R} \mathbb C\subset \mathfrak{g}$ denotes its complexification and $\Phi \subset {\mathfrak{h}}^{*}$ the root system of $\mathfrak{g}$ with respect to the Cartan subalgebra $\mathfrak{h}$.

\begin{definition}[Alekseevsky-David \cite{Liana}, Milburn \cite{Milburn}]
A \textit{{$\mathfrak{g}_0$}-admissible pair} is a pair $(\mathfrak{s},\omega )$, where $\mathfrak{s}\subset \mathfrak{g}$ is a complex subalgebra and $\omega \in \bigwedge ^2 {\mathfrak{s}}^{*}$ is a closed $2$-form such that:
\begin{itemize}
\item $\mathfrak{s}+\tau (\mathfrak{s})=\mathfrak{g}$, and
\item $\operatorname{Im}(\omega \mid _{{\mathfrak{g}_0}\cap \mathfrak{s}})$ is non-degenerate.
\end{itemize}
\end{definition}

\begin{theorem}[Akelseevsky-David \cite{Liana}, Milburn \cite{Milburn}]\label{theorem:1}
There is a one-to-one correspondence between the invariant generalized complex structures on $G_0$ and the ${\mathfrak{g}_0}$-admissible pairs $(\mathfrak{s},\omega )$.
\end{theorem}

The following notion was introduced by Alekseevsky and David \cite{Liana}. 

\begin{definition}[Akelseevsky-David \cite{Liana}]
An invariant generalized complex structure on $G_0$ is called \textit{regular} if the associated subalgebra $\mathfrak{s}\subset \mathfrak{g}$ is normalized by a Cartan subalgebra of $\mathfrak{g}_0$.
\end{definition}

The following theorem strengthens \cite{Liana}, Theorem~$15$ and completes the classification of invariant generalized complex structures on finite-dimensional compact connected real semisimple Lie groups.

\begin{theorem}\label{theorem:main}
If $G_0$ is a finite-dimensional compact connected real semisimple Lie group, then any invariant generalized complex structure on $G_0$ is regular.
\end{theorem}

\begin{proof}
Let $\mathfrak{s}\subset \mathfrak{g}$ be the complex subalgebra associated by Theorem~\ref{theorem:1} to an invariant generalized complex structure on $G_0$. Let $N (\mathfrak{s})\subset \mathfrak{g}$ be its normalizer.\\

By \cite{MalcevLarge}, Theorem~$13$, $N (\mathfrak{s})\cap \mathfrak{g}_0$ generates a closed subgroup of $G_0$. The same argument as in \cite{Liana}, Theorem~$15$, using \cite{Wang}, implies that $N (\mathfrak{s})$ is normalized by a Cartan subalgebra $\mathfrak{h}_0\subset \mathfrak{g}_0$, i.e. 
$$
N (\mathfrak{s}) = L\oplus \bigoplus\limits_{\alpha \in R}\mathbb C X_{\alpha},
$$
where $R\subset \Phi$ is a closed subset, $X_{\alpha}$, $\alpha \in \Phi$, are root vectors of $\mathfrak{g}$ with respect to the Cartan subalgebra $\mathfrak{h}=\mathfrak{h}_0 \otimes_{\mathbb R} \mathbb C$, and $L\subset \mathfrak{h}$ is the solution set of a system of equations of the form $\alpha -\beta =0$, $\alpha , \beta \in \Phi$.\\

Since $\tau$ is the conjugation with respect to a compact real form $\mathfrak{g}_0$ of $\mathfrak{g}$ and $\tau (\mathfrak{h})=\mathfrak{h}$,
$$
\tau \mid _{\mathfrak{h}(\mathbb R)}=-\operatorname{Id}_{\mathfrak{h}(\mathbb R)},\quad \tau (\mathbb C X_{\alpha})=\mathbb C X_{-\alpha},
$$
where $\mathfrak{h}(\mathbb R)$ is the real span in $\mathfrak{h}$ of the coroots of $\mathfrak{g}$ with respect to $\mathfrak{h}$ \cite{Helgason}.\\

Since $N (\mathfrak{s}) +\tau (N (\mathfrak{s}))=\mathfrak{g}$, $L +\tau (L)=\mathfrak{h}$, which is possible only if $L=\mathfrak{h}$. Hence $\mathfrak{h}_0\subset N(\mathfrak{s})$ normalizes $\mathfrak{s}$.

\end{proof}

In general, not all invariant generalized complex structures on real semisimple Lie groups are regular. Let $G_0$ be a finite-dimensional connected real Lie group, $k\geq 0$ an integer.

\begin{definition}
An invariant generalized complex structure $\mathcal J$ on $G_0$ is called \textit{subregular} \textit{in codimension $k$} if the associated subalgebra $\mathfrak{s}\subset \mathfrak{g}$ is normalized by a codimension $k$ subalgebra of a Cartan subalgebra of $\mathfrak{g}_0$.

If $k\geq 1$, then $\mathcal J$ is called \textit{subregular} \textit{strictly} \textit{in codimension $k$} if $\mathcal J$ is subregular in codimension $k$, but is not subregular in codimension $k-1$.
\end{definition}

We illustrate this notion with an example of a non-regular invariant generalized complex structure on $SO_0(2n-1,1)$, $n\geq 3$ even. 

\section{A non-regular invariant generalized complex structure on $SO_0(2n-1,1)$}

Let $G_0=SO_0(2n-1,1)$, $n\geq 3$. Then ${\mathfrak{g}_0}=\mathfrak{so}(2n-1,1)$ is a noncompact real form of $\mathfrak{g}=\mathfrak{so}_{2n}(\mathbb C)$. We interpret $\mathfrak{g}$ as the Lie algebra of $2n \times 2n$ skew symmetric complex matrices. Then 
$$
\tau \colon \mathfrak{g} \to \mathfrak{g}, \; A\mapsto J\cdot \bar{A}\cdot J,\quad J=\operatorname{diag}(\; \underbrace{1 \; 1 \; \cdots \; 1}_{2n-1} \; (-1)\; ),
$$
is the conjugation with respect to $\mathfrak{g}_0$, where bar denotes the usual complex conjugation.\\

Let $E_{ij}$, $1\leq i,j \leq 2n$, be a $2n\times 2n$ matrix with $1$ in the $(i,j)^{\text{th}}$ place and $0$ elsewhere. Following \cite{Helgason}, define
\begin{align*}
H_k & =\sqrt{-1}\cdot \left( E_{2k-1,2k} - E_{2k,2k-1} \right),\; 1\leq k\leq n,\\
G_{jk}^{+} & =E_{2j-1,2k-1} - E_{2k-1,2j-1}+E_{2j,2k} - E_{2k,2j}+\sqrt{-1}\cdot \left( E_{2j-1,2k} - E_{2j,2k-1}-E_{2k,2j-1} + E_{2k-1,2j} \right),\\
G_{kj}^{+} & =-\overline{G_{jk}^{+}},\; 1\leq j<k\leq n,\\
G_{jk}^{-} & =E_{2j-1,2k-1} - E_{2k-1,2j-1}-E_{2j,2k} + E_{2k,2j}+\sqrt{-1}\cdot \left( E_{2j-1,2k} + E_{2j,2k-1}-E_{2k,2j-1} - E_{2k-1,2j} \right),\\
G_{kj}^{-} & =-\overline{G_{jk}^{-}},\; 1\leq j<k\leq n.
\end{align*}

Then $\mathfrak{h}=\bigoplus\limits_{k=1}^{n} \mathbb C H_{k}$ is a Cartan subalgebra of $\mathfrak{g}$. Let ${\epsilon}_k \in {\mathfrak{h}}^{*}$, $1\leq k\leq n$, be such that ${\epsilon}_k (H_j)=1$ if $j=k$ and $0$ otherwise. Then
$$
\Phi = \{ {\epsilon}_j - {\epsilon}_k \; \mid \; 1\leq j\neq k\leq n  \} \cup \{ \pm ({\epsilon}_j + {\epsilon}_k) \; \mid \; 1\leq j< k\leq n  \}
$$
is the root system of $(\mathfrak{g},\mathfrak{h})$. Let us choose the following root vectors: 
\begin{align*}
X_{jk}=X_{{\epsilon}_j-{\epsilon}_k}=G_{jk}^{+},\; 1\leq j\neq k \leq n,\\
Y_{jk}=X_{{\epsilon}_j+{\epsilon}_k}=G_{kj}^{-},\; 1\leq j< k \leq n,\\
Z_{jk}=X_{-({\epsilon}_j+{\epsilon}_k)}=G_{jk}^{-},\; 1\leq j< k \leq n.
\end{align*} 

Note that 
$$
[Y_{jk},Z_{jk}]=4\cdot (H_j + H_k),\; [X_{jk},X_{kj}]=4\cdot (H_j - H_k),\; 1\leq j<k\leq n.
$$

Let $\mathfrak{h}_1\subset \mathfrak{h}$ be a hyperplane cut out by the equation ${\epsilon}_{n-1}-{\epsilon}_{n}=0$, $L\subsetneq \mathfrak{h}_1$ a vector subspace containing $H_{n-1}+H_n$, and $H\in \mathfrak{h}_1\setminus L$. Define
$$
\mathfrak{s}=L\oplus \mathbb C (H+X_{n-1,n})\oplus \bigoplus\limits_{\substack{1\leq j<k\leq n \\ (j,k)\neq (n-1,n)}} \mathbb C X_{jk}\;\oplus \bigoplus\limits_{1\leq j<k\leq n} \mathbb C Y_{jk}\;\oplus \; \mathbb C Z_{n-1,n}
$$
$$
\subset \mathfrak{g}=\mathfrak{h}\; \oplus \bigoplus\limits_{1\leq j\neq k\leq n} \mathbb C X_{jk}\; \oplus \bigoplus\limits_{1\leq j<k\leq n} \left( \mathbb C Y_{jk}\oplus \mathbb C Z_{jk} \right) .
$$

\begin{lemma}\label{lemma:22}
The subalgebra $\mathfrak{s}\subset \mathfrak{g}$ is subregular strictly in codimension $1$.
\end{lemma}

\begin{proof}

By construction, $\mathfrak{s}$ is normalized by a codimension $1$ subalgebra $\mathfrak{h}_1\subset \mathfrak{h}$. At the same time, $\mathfrak{s}$ is not regular, because, for a suitable $l\in L$, $l+H+X_{n-1,n}$ lies in the radical of $\mathfrak{s}$ but its nilpotent component $X_{n-1,n}$ does not. \end{proof}

Note that $s+\tau (s)=\mathfrak{g}$ if and only if $L+\mathbb C H + \tau (L+\mathbb C H)=\mathfrak{h}$.\\

To illustrate the general idea, let us assume for simplicity that $H=H_1$, $L=\bigoplus\limits_{k=2}^{n-2}\mathbb C H_k \; \oplus \mathbb C (H_{n-1}+H_n)$. Then 
$$
\mathfrak{s}\cap \mathfrak{g}_0 = \{\; \sqrt{-1}\cdot b_1 \cdot (H_1+X_{n-1,n}-Z_{n-1,n}) + \sum\limits_{j=2}^{n-2} \sqrt{-1}\cdot b_j \cdot H_j \; \mid \; b_j\in \mathbb R \;\}
$$

is a real abelian Lie algebra of dimension $n-2$. If $n$ is even, $\mathfrak{s}\cap \mathfrak{g}_0$ carries a symplectic form ${\omega}_0$, which may be any non-degenerate $2$-form on the real vector space $\mathfrak{s}\cap \mathfrak{g}_0=\mathbb R \sqrt{-1} (H_1+X_{n-1,n}-Z_{n-1,n})\;\oplus \bigoplus\limits_{j=2}^{n-2}\mathbb R \sqrt{-1}H_j\cong \mathbb R ^{n-2}$. One can extend $\sqrt{-1}{\omega}_0$ to a closed $2$-form $\omega \in \bigwedge ^2 {\mathfrak{s}}^{*}$. Assume that ${\epsilon}_j$, $1\leq j \leq n$, vanish on the root vectors of $\mathfrak{g}$. This proves

\begin{theorem}\label{theorem:22}
Let $n\geq 4$ be even, $\mathfrak{s}\subset \mathfrak{g}$ the complex subalgebra defined above, $H=H_1$, $L=\bigoplus\limits_{k=2}^{n-2}\mathbb C H_k \; \oplus \mathbb C (H_{n-1}+H_n)$, and $\omega \in \bigwedge ^2 {\mathfrak{s}}^{*}$ a closed $2$-form such that
$$
\omega \mid_{\mathfrak{s}\cap \mathfrak{g}_0}=\sqrt{-1}\cdot \sum\limits_{j=1}^{\frac{n}{2}-1} {\epsilon}_{2j-1}\wedge {\epsilon}_{2j}.
$$
Then $(\mathfrak{s},\omega)$ is a $\mathfrak{g}_0$-admissible pair and defines a non-regular invariant generalized complex structure on $SO_0(2n-1,1)$. 
\end{theorem}

\section{Invariant generalized complex structures on real forms of $G_2$}

In this section, $G_0$ denotes a connected real Lie group whose Lie algebra $\mathfrak{g}_0$ is a real form of the complex simple Lie algebra $\mathfrak{g}$ of type $G_2$, i.e. $\mathfrak{g}_0$ is either the compact real form $G_2^c$ or the normal real form $G_2^n$ of $\mathfrak{g}=G_2$. Let $\tau\colon \mathfrak{g}\to \mathfrak{g}$ be the conjugation with respect to $\mathfrak{g}_0$.\\

Recall that $G_2^n$ has $4$ conjugacy classes of Cartan subalgebras $\mathfrak{l}$: the maximally noncompact, the maximally compact, the one with a single short real root and the one with a single long real root \cite{Sugiura}. The conjugation ${\tau}_n \colon \mathfrak{g}\to \mathfrak{g}$ with respect to $G_2^n$ acts on the root system of $(\mathfrak{g},\mathfrak{l}\otimes_{\mathbb R}\mathbb C)$ as $\operatorname{Id}$, $-\operatorname{Id}$, a reflection through a short root and a reflection through a long root respectively.\\

Let $(\mathfrak{s},\omega)$ be a $\mathfrak{g}_0$-admissible pair corresponding to an invariant generalized complex structure on $G_0$. The subalgebras of the complex simple Lie algebra of type $G_2$ were classified in \cite{Mayanskiy}. We will use the notation of \cite{Dynkin} and \cite{Mayanskiy}. Since $\mathfrak{s}+\tau (\mathfrak{s})=\mathfrak{g}$, $\operatorname{dim} (\mathfrak{s})\geq \operatorname{dim}(G_2)/2=7$ and $\mathfrak{s}\subset \mathfrak{g}$ is regular.

\begin{lemma}\label{lemma:L}
The subalgebra $\mathfrak{s}\subset \mathfrak{g}$ is normalized by a Cartan subalgebra of $\mathfrak{g}_0$ and is not isomorphic to $\mathfrak{sl}_3(\mathbb C)$.
\end{lemma}

\begin{proof}

By \cite{Dynkin}, up to conjugacy either $\mathfrak{s}=\mathfrak{g}$ or $\mathfrak{s}=A_2$ or $\mathfrak{s}\subset G_2[\beta]$ or $\mathfrak{s}\subset G_2[\alpha]$.\\

Suppose $\mathfrak{s}=A_2$. Since $\mathfrak{s}$ is semisimple, the $2$-dimensional subalgebra $\mathfrak{s} \cap \tau (\mathfrak{s})$ contains semisimple and nilpotent components of its elements. Since $H^2(\mathfrak{s},\mathbb C )=0$, $\omega$ is exact. Thus, if $\mathfrak{s} \cap \tau (\mathfrak{s})$ is abelian, $\omega \mid_{\mathfrak{s}\cap \mathfrak{g}_0}=0$, a contradiction. Hence $\mathfrak{s} \cap \tau (\mathfrak{s})$ is not abelian, and so every element of $\mathfrak{s} \cap \tau (\mathfrak{s})$ is either semisimple or nilpotent. Then we can choose a basis $x_0, x_1$ of $\mathfrak{s} \cap \tau (\mathfrak{s})$ such that $[x_0,x_1]=2\cdot x_1$, where $x_0$~is semisimple and $x_1$~is nilpotent. Since $\tau (x_1)\in \mathbb C x_1$, we may assume that $\tau (x_0)=x_0$, $\tau (x_1)=x_1$.\\

The proof of the Jacobson-Morozov theorem in \cite{Orbits} goes through and provides $x_2\in \mathfrak{s}$ such that $x_0,x_1,x_2$ span an $\mathfrak{sl}_2(\mathbb C)$ subalgebra of $\mathfrak{s}$. Since $\tau (x_2)=x_2$, we obtain a contradiction.\\

Suppose $\mathfrak{s}\subset G_2[\beta]$ or $\mathfrak{s}\subset G_2[\alpha]$. By \cite{Mayanskiy}, Table~$1$, $\mathfrak{s}$ either is solvable and contains a Cartan subalgebra of $\mathfrak{g}$ or is normalized by a Borel subalgebra of $\mathfrak{g}$ or is the subalgebra
$$
\mathfrak{s}_3 = \mathfrak{h}_1\oplus \mathbb C Y_{\beta}\oplus \mathbb C Y_{-\beta}\oplus \mathbb C Y_{2\alpha +\beta}\oplus \mathbb C Y_{3\alpha +\beta}\oplus \mathbb C Y_{3\alpha +2\beta},
$$
where $\mathfrak{h}_1\subset \mathfrak{g}$ is a Cartan subalgebra, $\Phi = \{ \pm \alpha , \pm \beta , \pm (\alpha +\beta ) , \pm (2\alpha +\beta ) , \pm (3\alpha +\beta ) , \pm (3\alpha +2\beta ) \}$ is the root system of $(\mathfrak{g} , \mathfrak{h}_1)$, $Y_{\gamma}$, $\gamma\in \Phi$, are root vectors.\\ 

Note that any Borel subalgebra $\mathfrak{b}\subset \mathfrak{g}$ contains a Cartan subalgebra $\mathfrak{h}_0\subset \mathfrak{g}_0$ \cite{Wolf}. Let $\mathfrak{h} = \mathfrak{h}_0\otimes_{\mathbb R}\mathbb C$.\\

If $\mathfrak{s}\subset \mathfrak{b}$ contains a Cartan subalgebra of $\mathfrak{g}$, then $\mathfrak{h}_0$ is maximally compact and $\mathfrak{h} \cap \mathfrak{s} \neq 0$. This implies that either $\mathfrak{h}$ normalizes $\mathfrak{s}$ or $\mathfrak{h} \cap \mathfrak{s}\subset \mathfrak{s}\cap \tau (\mathfrak{s})=0$, a contradiction.\\

Suppose $\mathfrak{s}=\mathfrak{s}_3$. Let $\mathfrak{b}\subset G_2[\alpha]$ be the Borel subalgebra of $(\mathfrak{g},\mathfrak{h}_1)$ containing $Y_{\alpha}$ and $Y_{\beta}$, $\mathfrak{n}=[\mathfrak{b},\mathfrak{b}]$. Since $[[\mathfrak{n},\mathfrak{n}],\mathfrak{n}]\subset \mathfrak{s}$, we may write 
$$
\mathfrak{s} = \mathfrak{h}_1\oplus \mathbb C x_3 \oplus \mathbb C x_4 \oplus \mathbb C X_{2\alpha +\beta}\oplus \mathbb C X_{3\alpha +\beta}\oplus \mathbb C X_{3\alpha +2\beta},
$$
$$
x_3=a_0\cdot X_{\alpha}+a_1\cdot X_{\alpha +\beta}+ X_{\beta}, \quad x_4=X_{-\beta}+b_0\cdot X_{\alpha}+b_1\cdot X_{\alpha +\beta},
$$
where $X_{\gamma}$, $\gamma\in \Phi$, are root vectors of $(\mathfrak{g},\mathfrak{h})$, $\mathfrak{n}$ contains $X_{\alpha}$ and $X_{\beta}$.\\

We may assume that $\mathfrak{h}_1=\mathbb C x_1\oplus \mathbb C x_2$, where
$$
x_1=z_1+X_{\rho}+\sum\limits_{\gamma \succ \rho}u_{\gamma}\cdot X_{\gamma},\quad x_2=z_2+\sum\limits_{\gamma \succ \rho}v_{\gamma}\cdot X_{\gamma},\quad z_1,z_2\in \mathfrak{h},\quad \rho (z_2)=0,
$$
for some $\rho\in \{ \alpha , \beta , \alpha + \beta \}$.\\

Since $\mathfrak{s}\ni [x_2,x_3]=\alpha (z_2)\cdot a_0\cdot X_{\alpha}+\beta (z_2)\cdot X_{\beta}+x_{23}$, $x_{23}\in [\mathfrak{n},\mathfrak{n}]$, and $(\alpha -\beta )(z_2)\neq 0$, $a_0=0$. If $\rho \neq \alpha$, then also
$$
\mathfrak{s}\ni (\alpha +\beta )(z_2)\cdot a_1\cdot X_{\alpha +\beta}+\beta (z_2)\cdot X_{\beta},
$$

and so $a_1=0$ in this case.\\

If $x_3=X_{\beta}$, then $[x_3,x_4]=H_{\beta} +b_0\cdot X_{\alpha +\beta}$. Hence $\rho \neq \alpha $, and so $[x_2,x_4]\in \mathfrak{s}$ implies that $b_1=0$. Then
$$
\mathfrak{s}=\mathbb C x_1' \oplus \mathbb C x_2' \oplus \mathbb C x_4 \oplus \mathbb C X_{\beta} \oplus \mathbb C X_{2\alpha + \beta} \oplus  \mathbb C X_{3\alpha + \beta} \oplus \mathbb C X_{3\alpha + 2\beta},
$$
where $x_1'=z_1'+u\cdot X_{\alpha +\beta}$, $x_2'=z_2'+v\cdot X_{\alpha +\beta}$, $z_1',z_2'\in \mathfrak{h}$. We may assume that $u=1$, $v=0$, and so $(\alpha +\beta )(z_2')=0$.\\

In this case, $\tau$ acts on the roots either as $-\operatorname{Id}$ or as a reflection through $\beta$. Hence either $z_2'$ or $X_{\beta}$ is contained in $\mathfrak{s}\cap \tau (\mathfrak{s})=0$, a contradiction.\\

Hence we may assume that $x_3$ is not proportional to a root vector, and so $\rho = \alpha$. In this case, $\mathfrak{s}$ contains $x_1'=z_1+X_{\alpha}+u\cdot X_{\alpha +\beta}$ and $x_2'=z_2+v\cdot X_{\alpha +\beta}$.\\

If $v\neq 0$, we may assume that $v=1$ and $u=0$. Since $[x_1',x_2']\in \mathfrak{s}$, $(\alpha +\beta )(z_1)=0$.\\

Since $[x_1',x_4]\in \mathfrak{s}$, $-\beta (z_1)\cdot X_{-\beta}+\alpha (z_1)b_0\cdot X_{\alpha}\in \mathfrak{s}$, and so $b_1=0$.\\

Since $[x_2',x_4]\in \mathfrak{s}$, $-\beta (z_2)\cdot X_{-\beta}- X_{\alpha}\in \mathfrak{s}$, and so $b_0=1/ \beta (z_2)$.\\

Since $[x_1',x_3]\in \mathfrak{s}$, $\beta (z_1)\cdot X_{\beta}- X_{\alpha +\beta}\in \mathfrak{s}$, and so $a_1=-1/ \beta (z_1)$. Hence
$$
\mathfrak{s} = \mathbb C x_1'\oplus \mathbb C x_2'\oplus \mathbb C x_3'\oplus \mathbb C x_4'\oplus \mathbb C X_{2\alpha +\beta}\oplus \mathbb C X_{3\alpha +\beta}\oplus \mathbb C X_{3\alpha +2\beta},
$$ 
where $x_1'=r_1\cdot H_{3\alpha +\beta}-r_2\cdot X_{-\beta}$, $x_2'=r_2\cdot H_{3\alpha +2\beta}-r_1\cdot X_{\beta}$, $x_3'=X_{\alpha +\beta} +r_1\cdot X_{\beta}$, $x_4'=X_{\alpha} +r_2\cdot X_{-\beta}$.\\

Since $\tau$ acts on the roots either as $-\operatorname{Id}$ or as a reflection through $\beta$, $\mathfrak{h}\oplus \mathbb C X_{\beta}\oplus \mathbb C X_{-\beta}$ is spanned by $x_1'$, $x_2'$, $\tau (x_1')$, $\tau (x_2')$. We can choose the root vectors such that $\tau (X_{\gamma})=\pm X_{\tau (\gamma )}$, $\gamma \in \Phi$.\\

If $\tau$ acts on the roots as $-\operatorname{Id}$, then $\tau (x_1')=-\overline{r_1}\cdot H_{3\alpha +\beta}\mp \overline{r_2}\cdot X_{\beta}$, $\tau (x_2')=-\overline{r_2}\cdot H_{3\alpha +2\beta}\mp \overline{r_1}\cdot X_{-\beta}$. Hence $\mathbb C X_{\beta}\oplus \mathbb C X_{-\beta}$ is spanned by a single element $\left( r_2/r_1 \right) \cdot X_{-\beta}\pm \overline{\left( r_2/r_1 \right) } \cdot X_{\beta}$, a contradiction.\\

If $\tau$ acts on the roots as a reflection through $\beta$, then $\tau (x_1')=\pm\overline{r_1}\cdot H_{3\alpha +2\beta }+ \overline{r_2}\cdot X_{-\beta}$, $\tau (x_2')=\pm\overline{r_2}\cdot H_{3\alpha +\beta }+ \overline{r_1}\cdot X_{\beta}$. Hence $\mathfrak{h}$ is spanned by a single element $\left( r_1/r_2 \right) \cdot H_{3\alpha + \beta}\pm \overline{\left( r_1/r_2 \right) } \cdot H_{3\alpha +2\beta }$, a contradiction.\\

If $v=0$, then $u=0$. Since $[x_2',x_4]\in \mathfrak{s}$, $X_{-\beta}-b_1\cdot X_{\alpha +\beta}\in \mathfrak{s}$, and so $b_0=b_1=0$. Since $\mathfrak{s}\cap \tau (\mathfrak{s})=0$ and $\alpha (z_2)=0$, $\mathfrak{g}_0=G_2^n$ and $\tau = \tau _n$ acts on the roots as a reflection through $\beta$. Hence $X_{-\beta}\in \mathfrak{s}\cap \tau (\mathfrak{s})$, a contradiction. 

\end{proof}

\begin{corollary}\label{corollary:L}
The subalgebra $\mathfrak{s}\subset \mathfrak{g}$ is normalized by a maximally compact Cartan subalgebra $\mathfrak{h}_0\subset \mathfrak{g}_0$. Moreover, either $\mathfrak{s}=L\oplus \mathfrak{n}$ or $\mathfrak{s}=\mathfrak{b}$, where $\mathfrak{b}\subset \mathfrak{g}$ is a Borel subalgebra of $(\mathfrak{g},\mathfrak{h})$, $L\subset \mathfrak{h}=\mathfrak{h}_0\otimes_{\mathbb R}\mathbb C$ and $\mathfrak{n}=[\mathfrak{b},\mathfrak{b}]$.
\end{corollary}

\begin{proof}

Let $\mathfrak{h}_0\subset \mathfrak{g}_0$ be a Cartan subalgebra normalizing $\mathfrak{s}$, $\mathfrak{h}=\mathfrak{h}_0\otimes_{\mathbb R}\mathbb C$, $X_{\gamma}$, $\gamma \in \Phi$, root vectors of $(\mathfrak{g},\mathfrak{h})$.\\

By Lemma~\ref{lemma:L} and \cite{Mayanskiy}, Table~$1$, the subalgebra $\mathfrak{s}\subset \mathfrak{g}$ is one of the following:
$$
L\oplus \mathfrak{n},\quad \mathfrak{b}=\mathfrak{h}\oplus \mathfrak{n},\quad \mathbb C H_{\alpha}\oplus \mathbb C X_{-\alpha}\oplus \mathfrak{n},\quad \mathbb C H_{\beta}\oplus \mathbb C X_{-\beta}\oplus \mathfrak{n},
$$
$$
G_2[\beta]=\mathbb C X_{-\alpha}\oplus \mathfrak{b},\quad G_2[\alpha]=\mathbb C X_{-\beta}\oplus \mathfrak{b}, \quad \mathfrak{g},
$$
where $L\subset \mathfrak{h}$, $\mathfrak{n}=\bigoplus\limits_{\gamma \in {\Phi}^{+}}\mathbb C X_{\gamma}$ and $\Phi$ is suitably ordered so that ${\Phi}^{+} = \{ \alpha , \beta , \alpha +\beta , 2\alpha +\beta , 3\alpha +\beta , 3\alpha +2\beta \} \subset \Phi$ is the subset of positive roots.\\

By \cite{Liana}, Lemma~$7$, only two of these subalgebras can form a $\mathfrak{g}_0$-admissible pair:
$$
\mathfrak{s}=L\oplus \mathfrak{n} \quad \mbox{or}\quad \mathfrak{s}=\mathfrak{b}.
$$

In both cases, $\tau$ acts on the roots as $-\operatorname{Id}$.

\end{proof}

Let $\phi\colon G_0\to G$ be the universal complexification \cite{Hochschild}, i.e. $G$ is the connected complex simple Lie group of type $G_2$, with Lie algebra $\mathfrak{g}$, $\operatorname{ker}(\phi)$ is the center of $G_0$, and the differential of $\phi$ is the embedding $\mathfrak{g}_0\subset \mathfrak{g}$.\\

Let $B\subset G$ be a fixed Borel subgroup, with Lie subalgebra $\mathfrak{b}\subset \mathfrak{g}$ containing a maximally compact Cartan subalgebra $\mathfrak{h}_0\subset \mathfrak{g}_0$. By \cite{Wolf}, Theorem~$5.4$, $H_0=B\cap G_0$ is connected and so is generated by $\mathfrak{h}_0=\mathfrak{b}\cap \mathfrak{g}_0$.\\

Consider the flag manifold $G/B$ parametrizing the Borel subalgebras of $\mathfrak{g}$. Let $\mathcal{N}^{*}$ be the holomorphic homogeneous vector bundle over $G/B$ corresponding to the isotropy representation $B\to GL(\operatorname{Hom}_{\mathbb C}([\mathfrak{b},\mathfrak{b}],\mathbb C))$ coming from the adjoint action of $B$ on $\mathfrak{b}$.\\

Let $\mathcal{B}\subset G/B$ be the union of the open orbits of $G_0$. By \cite{Wolf}, Theorem~$4.5$, $\mathcal{B}$ parametrizes the Borel subalgebras containing a maximally compact Cartan subalgebra of $\mathfrak{g}_0$. If $G_0$ is compact, then $\mathcal{B}= G/B$. Otherwise, $\mathcal{B}$ consists of exactly three open orbits of $G_0$ on $G/B$, corresponding to the three Weyl chambers of $G_2$ contained in a Weyl chamber of $A_1+\tilde{A_1}$, \cite{Wolf}, Theorem~$4.7$.\\

Let 
$$
\mathcal{I}=\mathcal{B}\times GL(2,\mathbb R)/GL(1,\mathbb C), \quad \Sigma =\mathcal{B}\times {\Sigma}_0,\quad \mbox{where}\;\; {\Sigma}_0 = \{ \sigma \in \bigwedge ^2 ({\mathbb C }^{2})^{*}  \mid \operatorname{Im}(\sigma {\mid}_{{\mathbb R}^{2}}) \; \mbox{is symplectic} \},
$$
be the trivial bundles over $\mathcal{B}$ parametrizing the complex structures and certain extensions of symplectic structures on the fibers of $\mathcal{B}\times \mathfrak{h}_0\to \mathcal{B}$ respectively, $\mathfrak{h}_0$ identified with ${\mathbb R}^{2}$.\\

\begin{remark}
As sets, $GL(2,\mathbb R)/GL(1,\mathbb C) \cong \{ z\in \mathbb C \mid \operatorname{Im}(z)\neq 0 \} \cong {\Sigma}_0$.
\end{remark}

Now we state the main result of this section.

\begin{theorem}
Any invariant generalized complex structure on $G_0$, a real Lie group of type $G_2$ and real dimension $14$, is regular. The set of invariant generalized complex structures on $G_0$ is parametrized by the disjoint union
$$
\mathcal{C}\cup \mathcal{S},
$$
where $\mathcal{C}=\mathcal{I}\times_{\mathcal{B}}\mathcal{N}^{*}\cong \mathcal{N}^{*}{\mid}_{\mathcal{B}}\times GL(2,\mathbb R)/GL(1,\mathbb C)$ and $\mathcal{S}={\Sigma}\times_{\mathcal{B}}\mathcal{N}^{*}\cong \mathcal{N}^{*}{\mid}_{\mathcal{B}}\times {\Sigma}_0$.
\end{theorem}

\begin{proof}

We use the notation of Corollary~\ref{corollary:L}.\\

Suppose $\mathfrak{s}=L\oplus \mathfrak{n}$, $\operatorname{dim}(L)=1$. Then $\mathfrak{s}+\tau (\mathfrak{s})=\mathfrak{g}$ if and only if $L\subset \mathfrak{h}$ is the holomorphic subspace of a complex structure on $\mathfrak{h}_0$. Since $\mathfrak{s}\cap \tau (\mathfrak{s})=0$, any closed $2$-form $\omega \in \bigwedge ^2 {\mathfrak{s}}^{*}$ gives a $\mathfrak{g}_0$-admissible pair $(\mathfrak{s},\omega )$.\\

The Chevalley-Eilenberg resolution gives $H^2(\mathfrak{s},\mathbb C)=0$. Hence $\omega = \operatorname{d}\xi$ for a uniquely determined linear map $\xi \colon [\mathfrak{b},\mathfrak{b}]\to \mathbb C$.\\

Thus, the $\mathfrak{g}_0$-admissible pairs $(\mathfrak{s},\omega )$ with $\mathfrak{s}=L\oplus \mathfrak{n}$ are parametrized by the triples $(\mathfrak{b},\xi , \lambda)$, where $\mathfrak{b}\subset \mathfrak{g}$ is a Borel subalgebra containing a maximally compact Cartan subalgebra of $\mathfrak{g}_0$, $\xi \in \operatorname{Hom}_{\mathbb C}([\mathfrak{b},\mathfrak{b}],\mathbb C)$ and $\lambda \in GL(2,\mathbb R)/GL(1,\mathbb C)$ is a complex structure on the real vector space $\mathfrak{h}_0=\mathfrak{b}\cap \mathfrak{g}_0\cong {\mathbb R}^2$. Cf. \cite{Pittie}.\\

Suppose $\mathfrak{s}=\mathfrak{b}$. In this case, $H^2(\mathfrak{s},\mathbb C)=\mathbb C \cdot {\omega}_0$, where $0\neq {\omega}_0\in \bigwedge ^2 {\mathfrak{h}}^{*}$ is extended by zero to a $2$-form on $\mathfrak{s}$. Since $\mathfrak{b}\cap \mathfrak{g}_0=\mathfrak{h}_0$, a $2$-form $\omega\in \bigwedge ^2 {\mathfrak{b}}^{*}$ gives a $\mathfrak{g}_0$-admissible pair $(\mathfrak{b},\omega )$ if and only if $\omega = c\cdot {\omega}_0+\operatorname{d}\xi$ for a uniquely determined linear map $\xi \colon [\mathfrak{b},\mathfrak{b}]\to \mathbb C$, where $\operatorname{Im}(c\cdot \omega_0 {\mid}_{\mathfrak{h}_0})$ is non-degenerate.\\

Thus, the $\mathfrak{g}_0$-admissible pairs $(\mathfrak{s},\omega )$ with $\mathfrak{s}=\mathfrak{b}$ are parametrized by the triples $(\mathfrak{b},\xi , \sigma)$, where $\mathfrak{b}\subset \mathfrak{g}$ is a Borel subalgebra containing a maximally compact Cartan subalgebra of $\mathfrak{g}_0$, $\xi \in \operatorname{Hom}_{\mathbb C}([\mathfrak{b},\mathfrak{b}],\mathbb C)$ and $\sigma \in {\Sigma}_0$.

\end{proof}

As we recalled above, $\mathcal{B}$ consists of one or three orbits of $G_0$ \cite{Wolf}.

\begin{corollary}
The set of invariant generalized complex structures on $G_0$, up to conjugacy by $G_0$, is parametrized by $r$ copies of the disjoint union
$$
N_0^{*}\times GL(2,\mathbb R)/GL(1,\mathbb C)\cup N_0^{*}\times {\Sigma}_0,
$$
where $N_0^{*}=\operatorname{Hom}_{\mathbb C}([\mathfrak{b},\mathfrak{b}],\mathbb C)/H_0$, $r=1$ if $G_0$ is compact and $3$ otherwise.
\end{corollary}

The following remark is an immediate consequence of Milburn's study of invariant generalized complex structures on homogeneous spaces \cite{Milburn}.

\begin{remark}
There is no $SO(2n+1)$-invariant generalized complex structure on the $2n$-dimensional sphere $S^{2n}=SO(2n+1)/SO(2n)$, $n\geq 2$, and no $G_2$-invariant generalized complex structure on $S^{6}=G_2^c/SU(3)$. The $SO(3)$-invariant generalized complex structures on $S^{2}=SO(3)/SO(2)$ are two biholomorphic complex structures $\mathbb C \mathbb P ^1$ and $\overline{\mathbb C \mathbb P ^1}$, and a family of invariant generalized complex structures with holomorphic subbundles of the form $L(\mathfrak{so}_{3}(\mathbb C), {\omega}_c)$, $c\in \mathbb C$, $\operatorname{Im} (c)\neq 0$, which are B-transforms of the symplectic structures (up to symplectomorphism) on $S^2$. Notation is from \cite{Gualtieri}, \cite{Milburn}, ${\omega}_c \in \bigwedge ^2 {\mathfrak{so}_{3}(\mathbb C)}^{*}$ is defined by 
$$
{\omega}_c (X,Y)=\operatorname{Trace}(\begin{pmatrix} 0 & c & 0 \\ -c & 0 & 0 \\ 0 & 0 & 0 \end{pmatrix}\cdot [X,Y])
$$
for $3\times 3$ skew symmetric complex matrices $X, Y \in \mathfrak{so}_3(\mathbb C)$. See also \cite{Hitchin}.
\end{remark}

\section*{Acknowledgement}

The author is grateful to Beijing International Center for Mathematical Research, the Simons Foundation and Peking University for support, excellent working conditions and encouraging atmosphere.\\

\bibliographystyle{ams-plain}

\bibliography{SubregInvGCS}

\end{document}